\documentclass[12pt]{amsart}
\usepackage{amscd,amsmath,amsthm,amssymb,graphics,color}
\usepackage{pstcol,pst-plot,pst-3d}
\usepackage{lmodern,pst-node}
\usepackage[dvips]{graphicx}
\usepackage{multicol}
\usepackage{epic,eepic}
\usepackage{amsfonts,amssymb,amscd,amsmath,enumitem,verbatim}
\psset{unit=0.7cm,linewidth=0.8pt,arrowsize=2.5pt 4}

\newpsstyle{fatline}{linewidth=1.5pt}
\newpsstyle{fyp}{fillstyle=solid,fillcolor=verylight}
\definecolor{verylight}{gray}{0.97}
\definecolor{light}{gray}{0.9}
\definecolor{medium}{gray}{0.85}
\definecolor{dark}{gray}{0.6}

\usepackage{lineno}

\def\NZQ{\Bbb}               

\def\ZZ{{\NZQ Z}}

\def\frk{\frak}               

\def\Phi{{\frk n}}
\def\Phi{{\frk N}}
\def\MC{{\mathcal C}}

\def\opn#1#2{\def#1{\operatorname{#2}}} 
\opn\chara{char} \opn\length{\ell} \opn\pd{pd} \opn\rk{rk}
\opn\projdim{proj\,dim} \opn\injdim{inj\,dim} \opn\rank{rank}
\opn\depth{depth} \opn\grade{grade} \opn\height{height}  \opn\bigheight{bigheight}
\opn\embdim{emb\,dim} \opn\codim{codim}

\opn\Tr{Tr} \opn\bigrank{big\,rank}
\opn\superheight{superheight}\opn\lcm{lcm}
\opn\trdeg{tr\,deg}
\opn\reg{reg} \opn\lreg{lreg} \opn\ini{in} \opn\lpd{lpd}
\opn\size{size}\opn\bigsize{bigsize}
\opn\cosize{cosize}\opn\bigcosize{bigcosize}
\opn\sdepth{sdepth}\opn\sreg{sreg}
\opn\link{link}\opn\fdepth{fdepth}
\opn\cdeg{cdeg}
\opn\div{div} \opn\Div{Div} \opn\cl{cl} \opn\Cl{Cl}
\let\epsilon\varepsilon
\let\phi=\varphi
\let\kappa=\varkappa
\opn\Spec{Spec} \opn\Supp{Supp} \opn\supp{supp} \opn\Sing{Sing}
\opn\Ass{Ass} \opn\Min{Min}\opn\Mon{Mon} \opn\dstab{dstab} \opn\astab{astab}
\opn\Syz{Syz}
\opn\Ann{Ann} \opn\Rad{Rad} \opn\Soc{Soc}
\opn\Im{Im} \opn\Ker{Ker} \opn\Coker{Coker} \opn\Am{Am}
\opn\Hom{Hom} \opn\Tor{Tor} \opn\Ext{Ext} \opn\End{End}
\opn\Aut{Aut} \opn\id{id}

\opn\nat{nat}
\opn\pff{pf}
\opn\Pf{Pf} \opn\GL{GL} \opn\SL{SL} \opn\mod{mod} \opn\ord{ord}
\opn\Gin{Gin} \opn\Hilb{Hilb}\opn\sort{sort}
\opn\initial{init}
\opn\ende{end}
\opn\height{height}
\opn\type{type}
\opn\set{set}
\opn\indeg{indeg}
\opn\aff{aff} \opn\con{conv} \opn\relint{relint} \opn\st{st}
\opn\lk{lk} \opn\cn{cn} \opn\core{core} \opn\vol{vol}
\opn\link{link} \opn\star{star}\opn\lex{lex}\opn\cochord{co-chord}\opn\im{im}
\opn\gr{gr}

\def\pot#1#2{#1[\kern-0.28ex[#2]\kern-0.28ex]}

\opn\dirlim{\underrightarrow{\lim}}
\opn\inivlim{\underleftarrow{\lim}}
\let\to=\rightarrow

\def\Implies{\ifmmode\Longrightarrow \else
        \unskip${}\Longrightarrow{}$\ignorespaces\fi}
\def\implies{\ifmmode\Rightarrow \else
        \unskip${}\Rightarrow{}$\ignorespaces\fi}
\def\iff{\ifmmode\Longleftrightarrow \else
        \unskip${}\Longleftrightarrow{}$\ignorespaces\fi}

\let\:=\colon
 \theoremstyle{plain}
\newtheorem{Theorem}{Theorem}[section]
 \newtheorem{Lemma}[Theorem]{Lemma}
 \newtheorem{Corollary}[Theorem]{Corollary}
 \newtheorem{Proposition}[Theorem]{Proposition}

 \theoremstyle{definition}
 \newtheorem{Definition}[Theorem]{Definition}
 \newtheorem{Remark}[Theorem]{Remark}

 \newtheorem{Examples}[Theorem]{Examples}
 
\let\epsilon\varepsilon
\let\kappa=\varkappa
\opn\dis{dis}
\def\pnt{{\raise0.5mm\hbox{\large\bf.}}}

\opn\Lex{Lex}

\begin{document}
\title{Sequentially Cohen-Macaulay binomial edge ideals of closed graphs}
\author {Viviana Ene, Giancarlo Rinaldo, Naoki Terai}

\dedicatory
{Dedicated to J\"urgen Herzog on the occasion of his $80^{th}$ birthday}

\address{Viviana Ene, Faculty of Mathematics and Computer Science, Ovidius University, Bd. Mamaia 124, 900527 Constanta, Romania}  \email{vivian@univ-ovidius.ro}

\address{Giancarlo Rinaldo,  Department of Mathematics, Informatics, Physics and Earth
science, University of Messina, Viale F. Stagno d'Alcontres, 31, 98166 Messina,
Italy}  \email{giancarlo.rinaldo@unime.it}

\address{Naoki Terai, Department of Mathematics, Okayama University, 3-1-1, Tsushima-naka, Kita-ku, Okayama, 700-8530, Japan}  \email{terai@okayama-u.ac.jp}

\thanks{The third author was supported by the JSPS Grant-in Aid for Scientific Research (C)  18K03244. }

\begin{abstract}
In this paper we provide a full
combinatorial characterization of sequentially Cohen-Macaulay binomial edge ideals of
closed graphs. In addition, we show that a binomial edge ideal of a closed graph is approximately Cohen-Macaulay if and only if it is almost Cohen-Macaulay. 
\end{abstract}


\maketitle

\section*{Introduction}
\label{S:Intro}

Binomial edge ideals arising from finite simple graphs were introduced in 2010 independently in two papers, \cite{HHHKR} and 
\cite{Ohtani}. Since then, they have become a topic of intensive research in combinatorial commutative algebra. Apart of the applications of these ideals in algebraic statistics, there is a great interest in their investigation since it turned out that fundamental algebraic and homological properties of these ideals  are strongly related to the combinatorial data of the underlying graphs. 

Let $G$ be a simple graph, that is, undirected and without loops or multiple edges, on the vertex set $[n]$. Let 
$S=K[x_1,x_2,\ldots,x_n,y_1,y_2,\ldots,y_n]$ be the polynomial ring in $2n$ variables over a field $K.$ For $1\leq i<j\leq n,$ we set 
$f_{ij}=x_iy_j-x_jy_i.$ The binomial edge ideal $J_G$ of the graph $G$ is generated by all the binomials $f_{ij}$ where $i<j$ and $\{i,j\}$ an edge in 
$G.$ In particular, for the complete graph on $n$ vertices, $J_G$ is the determinantal ideal generated by all the maximal minors of the $2\times n$-generic matrix $X$ with the rows $x_1,x_2,\ldots,x_n$ and $y_1,y_2,\ldots,y_n.$

Basic facts about  binomial edge ideals and their minimal free resolutions are presented in \cite[Chapter 7]{HHO}. Cohen-Macaulay binomial edge ideals were investigated in several papers; see, for example, \cite{BMS, BMS2, EHH, KM, RaRi, Ri}. Gorenstein binomial edge ideals are classified in \cite{GM}. Licci binomial edge ideals are characterized in \cite{ERT}.  A Hochster type formula for the local cohomology modules of binomial edge ideals is obtained in \cite{AM}.

In this paper, we focus on the study of sequentially Cohen-Macaulay  binomial edge ideals. Some  particular classes of graphs with sequentially Cohen-Macaulay binomial edge ideals are studied in \cite{Za}. A complete classification of the graphs whose binomial edge ideals are sequentially Cohen-Macaulay seems to be very difficult. Even the full classification of graphs whose binomial edge ideals are Cohen-Macaulay is not yet known. However, a significant progress in this direction has been recently done in \cite{BMS2}. 

As in the case of the Cohen-Macaulay property, for studying sequentially Cohen-Macaulay binomial edge ideals, we can reduce to connected  indecomposable graphs; see  Remark~\ref{R:disconectSCM}. 

In this paper we focus on binomial edge ideals with quadratic Gr\"obner bases, that is, those ideals which are associated with closed graphs, introduced in 
\cite{HHHKR}. There are several combinatorial characterizations of closed graphs. We will especially use the one given in \cite[Theorem 2.2]{EHH} which states that a graph $G$ is closed if and only if the facets of the clique complex $\Delta(G)$ are intervals. In fact, one can order the facets of $\Delta(G)$, say 
$F_1,\ldots, F_r,$ such that $F_i=[a_i,b_i]:=\{m\in \ZZ: a_i\leq m\leq b_i\}$ for $1\leq i\leq r$ and $1=a_1<a_2<\cdots<a_r<b_r=n.$ By using this characterization, 
in Proposition~\ref{prop:cutsetsclosed} we specify all the minimal prime ideals of $J_G$ when $G$ is closed.  In Proposition~\ref{Prop:2cliques}, we show that all indecompo\-sable closed graphs with two maximal cliques are sequentially Cohen-Macaulay. The main result of this article  is Theorem~\ref{Th:SCMClosed} which states that if $G$ is an indecomposable closed graph with the clique complex $\Delta(G)=\langle F_1,F_2,\ldots,F_{r+1}\rangle, r\geq 2,$
where $F_i=[a_{i-1},b_i]$ for $1\leq i\leq r+1$ and $1=a_0<a_1<\ldots<a_r,$ $b_1<\ldots<b_{r+1}=n,$
then $J_G$ is sequentially Cohen-Macaulay if and only if there exists $1\leq k\leq r$ such that $a_i=a_r-(r-i)$ for $k \leq i \leq r$ and $b_j=b_1+(j-1)$ for $1\leq j \leq k.$ The most part of Section~\ref{S:closed} is devoted to the proof of this theorem.  A result of Goodarzi \cite{Goo} providing a characterization of
homogeneous sequentially Cohen-Macaulay ideals
 has a crucial role in the proof. For the necessity of the numerical conditions in 
Theorem~\ref{Th:SCMClosed}, we use a localization argument together with Goodarzi's condition. For the sufficiency, which is the most technical part of the proof, we proceed by induction on the number of the maximal cliques of $G$ in order to verify  Goodarzi's condition for $J_G$.
The  behavior of associated primes of a homogeneous ideal under hyperplane sections, which follows from \cite[Section 3]{MVY}, plays a key role in our arguments. 

Well known examples of sequentially Cohen-Macaulay modules are approximately Cohen-Macaulay modules. In Section~\ref{S:Approx}, we give the classification of closed graphs whose binomial edge ideals are approximately Cohen-Macaulay. We show in 
Theorem~\ref{thm:approxCMclosed} that, if $G$ is closed,  $J_G$ is is approximately Cohen-Macaulay if and only if $J_G$ is almost Cohen-Macaulay.
Corollary~\ref{cor:almostCMclosed} classifies the closed graphs $G$ with the property that $J_G$ is almost Cohen-Macaulay.

\section{Preliminaries}
\label{S:Prelim}

\subsection{Binomial edge ideals and closed graphs}
\label{SubS:BEI} 
All graphs considered in this article are \emph{simple}, i.e. undirected without loops or multiple edges.
Let $G$ be a simple graph on the vertex set $[n]:=\{1,2,\ldots,n\}$  with the edge set 
$E(G).$  Let $K$ be a field and $S=K[x_1,x_2,\ldots,x_n,y_1,y_2,\ldots,y_n]$ the polynomial $K$-algebra in $2n$ variables. The \emph{binomial edge ideal} $J_G$ is generated by all the binomials $f_{ij}:=x_iy_j-x_jy_i$ with $i<j$ and $\{i,j\}\in E(G).$ We consider the ring $S$ endowed with the lexicographic order induced by $x_1>x_2>\cdots >x_n>y_1>y_2>\cdots >y_n.$ By \cite[Corollary~2.2]{HHHKR}, $J_G$ is a radical ideal. Therefore, $J_G=\bigcap_{P\in\Min(J_G)}P,$ where $\Min(J_G)$ denotes the set of the minimal prime ideals of $J_G$.  
In \cite[Section 3]{HHHKR}, the elements of $\Min(J_G)$ are characterized in terms of the combinatorics of the graph $G.$ Let $W\subset [n]$ be a (possibly empty) subset of $[n]$, and let $G_1,\ldots,G_{c(W)}$ be the connected components of $G\setminus W$ where 
$G\setminus W$ is the induced subgraph of $G$ on the vertex set $[n]\setminus W.$ For $1\leq i\leq c(W),$ let $\tilde{G}_i$ be the complete 
graph on the vertex set $V(G_i).$ Let \[P_{W}(G)=(x_i,y_i: i\in W) +J_{\tilde{G_1}}+\cdots +J_{\tilde{G}_{c(W)}}.\]
Then $P_{W}(G)$ is a prime ideal of $\dim S/P_W(G)=n-|W|+c(W),$ for every $W\subset [n]$ \cite[Lemma~3.1]{HHHKR}.

By \cite[Theorem 3.2]{HHHKR}, $J_G=\bigcap_{W\subset [n]}P_{W}(G).$ Therefore, the minimal primes of $J_G$ are among the prime ideals $P_{W}(G)$ with $W\subset [n].$ The next known result provides a
precise classification of minimal primes of binomial edge ideals.

\begin{Proposition}\label{cpset}\cite[Corollary 3.9]{HHHKR}
$P_{W}(G)$ is a minimal prime of $J_G$ if and only if either $W=\emptyset$ or $W$ is non-empty and, for each 
$i\in W,$ $c(W\setminus\{i\})<c(W)$.
\end{Proposition}
A  non-empty set $W\subset [n]$ such that $P_{W}(G)$ is a minimal prime ideal of $J_G$ is called a \emph{cut set} of $G.$ Let 
$\MC(G)$ be the collection of all cut  sets of $G$ together with the empty set. With this notation, we get
\begin{equation}\label{eq:dimension}
\dim S/J_G=\max\{n+c(W)-|W|: W \in \MC(G)\},
\end{equation} where $\dim$ denotes the Krull dimension. In addition, it follows that  $J_G$ is unmixed, that is, all its minimal prime ideals have the same dimension, if and only if  $n+c=\dim S/P_{\emptyset}(G)=\dim S/P_W(G),$ for all $W \in \MC(G),$ which is equivalent to $c(W)=|W|+c$ for all $W \in \MC(G).$
 
In what follows, we are interested in binomial edge ideals of closed graphs, introduced in \cite{HHHKR}. They are also known in combinatorics as \emph{proper interval graphs}, see \cite{CR2}. By \cite[Proposition 1.2]{HHHKR} they are a subclass
of chordal graphs, those with no induced cycle of length greater than $3$. There are several
characterizations of closed graphs in literature. One of them involves the clique
(simplicial) complex $\Delta(G)$ of $G$, whose faces are the \emph{cliques} of $G$, i.e. complete induced subgraphs of $G.$

\begin{Theorem}\label{th:closedchar}\cite{EHH, HHHKR}
Let $G$ be a graph on the vertex set $[n].$ The following statements are equivalent:
\begin{itemize}
\item[\emph{(i)}] $G$ is a closed; 
\item[\emph{(ii)}] The generators of $J_G$ form a Gr\"obner basis with respect to the lexicographic order induced by $x_1>\cdots >x_n>y_1>\cdots>y_n;$
\item [\emph{(iii)}] The facets, say $F_1,\ldots, F_r,$ of the clique complex $\Delta(G)$ of $G$ are intervals of the form $F_i=[a_i,b_i]$ which can be ordered such that $1=a_1<\cdots < a_r<b_r=n.$
\end{itemize}
\end{Theorem}

In what follows, for a closed graph $G$, we consider a labeling of its vertices such that the maximal cliques $F_1,F_2,\ldots,F_r$ are intervals as in 
condition (iii) of the above theorem. Note that this is a leaf order of the facets of $\Delta(G).$
As usual,  if $F_1,\ldots, F_r$ are all the facets of $\Delta(G)$, we write 
$\Delta(G)=\langle F_1,\ldots,F_r\rangle.$

By using condition (iii) in the above theorem, we may characterize the cut sets of a closed graph, but before we need a definition. 

\begin{Definition}
Let $G$ be a connected closed graph with $\Delta(G)=\langle F_1,\ldots,F_r\rangle$ with $F_i=[a_i,b_i]$ such that 
$1=a_1<\cdots < a_r<b_r=n.$ Let $W_i=F_i\cap F_{i+1}$ for $1\leq i\leq r-1.$  We call the sets $W_i$ the 
\emph{connected cut sets} of $G.$
\end{Definition}

 Clearly, $W_i$ are cut sets of $G.$ Indeed, we have 
$c(W_i)=2$ and the connected components of $G\setminus W_i$ are $(F_1\cup\cdots\cup F_i)\setminus W_i$ and 
$(F_{i+1}\cup\cdots \cup F_r)\setminus W_i.$ If $j\in W_i,$ then there exists a path connecting a vertex from $(F_1\cup\cdots\cup F_i)\setminus W_i$, the vertex $j,$ and a vertex from $(F_{i+1}\cup\cdots \cup F_r)\setminus W_i.$ Therefore, 
$c(W_i\setminus\{j\})=1< c(W_i).$ 

\begin{Proposition}\label{prop:cutsetsclosed}
Let $G$ be a connected closed graph and a  non-empty set $W\subset [n].$ Then the following are equivalent:
\begin{itemize}
	\item [\emph{(a)}] $W$ is a cut set of $G;$
	\item [\emph{(b)}] $W=W_{j_1}\cup \cdots \cup W_{j_t}$ for some $t\geq 1, 1\leq j_1<\cdots <j_t\leq r-1,$ where $W_{j_i}$ is a connected cut set of $G$ 
		for every $i,$ and $\max(W_{j_i})+1<\min W_{j_{i+1}}$ for 
$1\leq i\leq t-1.$
\end{itemize}
Moreover, $\dim S/P_W(G)\leq n+1$ with equality if and only if $|W_{j_i}|=1,$ for every $1\leq i\leq t.$
\end{Proposition}

\begin{proof} Let $\Delta(G)=\langle F_1,\ldots,F_r\rangle$ and $F_i=[a_i,b_i]$ such that 
$1=a_1<\cdots < a_r<b_r=n.$
Let $W=W_{j_1}\cup \cdots \cup W_{j_t}$ as in the statement. Then $c(W)=t+1.$ Let $G_1,\ldots,G_{t+1}$ be the connected components of 
$G\setminus W,$ where $\max V(G_k)<\min V(G_{k+1})$ for $1\leq k\leq t.$ Let $i\in W.$ There exists $1\leq \ell\leq t$ such that 
$i\in W_{j_\ell}.$ Then there exists a path connecting a vertex of $G_\ell,$ the vertex $j$, and a vertex of $G_{\ell+1}.$ Therefore, 
$c(W\setminus\{i\})=t<c(W).$ This shows that $W$ is a cut set of $G.$

For the converse, we proceed by induction on $r.$ For $r=2,$ the claim is clear since  the only non-empty cut set of $G$ is 
$W_1=F_1\cap F_2.$ Let $r>2$ and let $W$ be a non-empty cut set of $G.$ We consider two cases.

\emph{Case 1.} $W\supseteq W_1.$ If $W=W_1,$ we are done. Let $W\supsetneq W_1.$ We have $P_{W}(G)=(x_i,y_i\: i\in W) +J_{\tilde{G_1}}+\cdots +J_{\tilde{G}_{c_G(W)}}$ where 
$G_1,\ldots,G_{c_G(W)}$ are the connected components of $G\setminus W.$ We can write
\[
P_{W}(G)=(x_i,y_i:i\in W_1)+(x_i,y_i: i\in W\setminus W_1) +J_{\tilde{G_1}}+\cdots +J_{\tilde{G}_{c_G(W)}}\] 
\[
=(x_i,y_i:i\in W_1)+P_{W\setminus W_1}(G\setminus W_1).
\]
Let us observe that $P_{W\setminus W_1}(G\setminus W_1)$ is a minimal prime of $G\setminus W_1,$ or, equivalently, $W\setminus W_1$
is a cut set of $G\setminus W_1.$ Indeed, $c_{G\setminus W_1}(W\setminus W_1)>c_{G\setminus W_1}((W\setminus W_1)\setminus\{j\})$ for all
$j\in W\setminus W_1,$ since, otherwise, $c_G(W)\leq c_G(W\setminus\{j\}).$ The graph $G\setminus W_1$ consists of two connected components. The first one is the clique on the set $F_1\setminus W_1$ and the second one is a closed graph whose maximal cliques are 
the maximal elements in the set $\{F_q\setminus W_1, F_{q+1},\ldots,F_r\}$ where the index $q$ is defined by the condition 
$a_q\leq \max W_1<a_{q+1}.$ In fact, if $a_{q+1}=\max W_1+1,$ then $\Delta(G\setminus W_1)=\langle F_1\setminus W_1,F_{q+1},\ldots,F_r\rangle$
and if $a_{q+1}>\max W_1+1,$ then $\Delta(G\setminus W_1)=\langle F_1\setminus W_1,F_q\setminus W_1,F_{q+1},\ldots,F_r\rangle.$

By induction, $W\setminus W_1=W'_{j_1}\cup \cdots \cup W'_{j_s}$ for some $s\geq 1$, where $W'_{j_i}$ are connected cut sets of $G\setminus W_1$ and of $G$ as well, for $1\leq i\leq s.$ In addition, by the inductive hypothesis, the sets 
$ W'_{j_i}$ satisfy the conditions of the statement. Thus, we only need to show that $\max W_1+1<\min W'_{j_1},$ but this is clear since 
$\min W'_{j_1}> a_{q+1}$ if $a_{q+1}=\max W_1+1$ and $\min W'_{j_1}= a_{q+1}$ if $a_{q+1}>\max W_1+1.$

\emph{Case 2.} $W\not\supset W_1.$ In this case we show that $W\subset W_2\cup \cdots\cup W_{r-1}.$ Then it will  follow that $W$ is a cut set 
of the connected closed graph $G'$ with $\Delta(G')=\langle F_1\cup F_2,F_3,\ldots,F_r\rangle.$ Therefore, by induction on $r$, we get the desired form of $W.$ To show that $W\subset W_2\cup \cdots\cup W_{r-1}$ we proceed by contradiction. Let us assume that 
$W\not\subset  W_2\cup \cdots\cup W_{r-1}.$ Then there exists some vertex $i\in W_1\setminus (W_2\cup \cdots\cup W_{r-1}).$ On the other hand, 
there exists a vertex $j\in W_1\setminus W$ since $W\not\supset W_1.$ 
Since $j\in W_1=F_1\cap F_2$, we have $c_G(W\setminus\{i\})=c(W),$ a contradiction.
For the last part of the statement, we observe that  $\dim S/P_W(G)=n+c(W)-|W|=n+(t+1)-\sum_{i=1}^t |W_{j_i}|\leq n+1$ with equality if and only if $\sum_{i=1}^t |W_{j_i}|=t,$ that is, if and only if $|W_{j_i}|=1$ for all $i.$
\end{proof}

For a vertex $v$ of a graph $G,$ we denote by $\cdeg(v)$ the \emph{clique degree} of $v,$ that is, the number of cliques to which the vertex $v$ belongs. In particular, for a tree,
the clique degree of $v$ is the ordinary degree. The vertex $v$ of $G$ is called a \emph{free vertex} if $\cdeg(v)=1.$ Otherwise, $v$ is called an \emph{inner vertex}. According to \cite[Definition 1.2]{HeRi}, a connected graph $G$ is called \emph{decomposable} if there exists two subgraphs 
$G_1$, $G_2$ of $G$ such that $G=G_1\cup G_2$ with $V(G_1)\cap V(G_2)=\{v\}$ where $v$ is a free vertex in $G_1$ and $G_2.$ If $G$ is not decomposable, then $G$ is called an \emph{indecomposable graph}.  Clearly, every connected graph $G$ has a unique (up to ordering) decomposition $G=G_1\cup G_2\cup\cdots \cup G_m$ where $G_1,\ldots,G_m$ are indecomposable graphs and, for $1\leq i<j\leq m,$ either 
$V(G_i)\cap V(G_j)=\emptyset$ or $V(G_i)\cap V(G_j)=\{v\}$ and $v$ is a free vertex in $G_i$ and $G_j.$ 

\begin{Remark}\label{R:decompose}
Let $G$ be a connected closed graph on the vertex set $[n].$ Then $J_G$ is unmixed if and only if there exist integers $1=a_1<a_2<\cdots <a_r<a_{r+1}=n$ such that  the facets $F_1,\ldots,F_r$ of $\Delta(G)$ are of the form $F_i=[a_i,a_{i+1}]$ for all $i=1,\ldots,r$. Indeed, as we have noticed above, $J_G$ is unmixed  if and only if   $c(W)=|W|+1$ for all $W \in \MC(G).$ In particular, it follows that all the connected cut sets of $G$ have cardinality $1.$
On the other hand, By \cite[Proposition 1.3]{HeRi} or \cite{RaRi},
if a graph $G$ is decomposable as $G=G_1\cup G_2\cup\cdots \cup G_m$, then 
$J_G$ is Cohen-Macaulay if and only if $J_{G_i}$ is Cohen-Macaulay for all $\leq i \leq m.$
In particular, for a connected closed graph $G$, we may decide whether $J_G$ is Cohen-Macaulay or not by looking at the facets of the clique complex $\Delta(G).$ More precisely, $J_G$ is Cohen-Macaulay if and only if the facets of $\Delta(G)$ are intervals of the form 
$F_i=[a_i,a_{i+1}]$ for all $i=1,\ldots,r$ with  $1=a_1<a_2<\cdots <a_r<a_{r+1}=n$. Indeed, if $J_G$ is Cohen-Macaulay, then $J_G$ is unmixed and the claim follows. Conversely, if the facets of $\Delta(G)$ are intervals as above, then $G$ is decomposable in $r$ complete graphs, thus $J_G$ is Cohen-Macaulay. Note that this characterization of Cohen-Macaulay binomial edge ideals of closed graphs was proved in \cite[Theorem~3.1]{EHH}. 
\end{Remark}



\subsection{Sequentially Cohen-Macaulay ideals}
\label{SubS:SCMideals}

In this section we review known facts about sequentially Cohen-Macaulay graded ideals that we are going to use in the further sections.

\begin{Definition}\label{Def:SCMmodules}
Let $M$ be a finitely generated graded module over a standard graded polynomial $K$-algebra $R=K[x_1,\ldots,x_n].$ The module $M$ is called \emph{sequentially Cohen-Macaulay} if there exists a finite filtration with graded submodules 
\[
0=M_0\subset M_1\subset \cdots \subset M_r=M
\]
such that the following two conditions are fulfilled: 
\begin{itemize}
	\item [(1)] $M_i/M_{i-1}$ is a Cohen-Macaulay module for $1\leq i\leq r;$
	\item [(2)] $\dim(M_1/M_0)<\dim(M_2/M_1)<\cdots <\dim(M_r/M_{r-1}).$
\end{itemize}
\end{Definition}

Sequentially Cohen-Macaulay modules were introduced by Stanley \cite{Sta}. In the original definition \cite[Definition 2.9]{Sta}, there were considered graded modules over arbitrary graded algebras, but in what follows we may restrict to graded modules over standard graded polynomial rings. Later on, several authors have considered sequentially Cohen-Macaulay modules over local rings; see, for example, 
\cite{CN, Sch}.

\begin{Remark}\label{R:disconectSCM} We may reduce the study of the sequentially Cohen-Macaulay property of binomial edge ideals to connected graphs. This is due to \cite[Theorem~2.11]{STY} and to the decomposition of $S/J_G$ as a tensor product as
\[
\frac{S}{J_G}\cong \frac{S_1}{J_{G_1}}\otimes \frac{S_2}{J_{G_2}}\otimes \cdots \otimes \frac{S_c}{J_{G_c}},
\] if $G_1,G_2,\ldots,G_c$ are the connected components of $G,$ where $S_i=K[\{x_j,y_j: j\in V(G_i)\}]$ for $1\leq i\leq c.$  

Moreover, we can even reduce to indecomposable graphs. Indeed, if $G$ is decomposable as $G=G_1\cup G_2$ along the vertex $v,$ then, 
by the proof of  \cite[Theorem~2.7]{RaRi}, we may consider a new graph $G'$  with the connected components $G_1$ and 
$G_2'$ where $G_2'$ is obtained as follows. We remove the vertex $v$ from $G_2$ and add a new vertex $v',$ thus 
$V(G_2')=(V(G_2)\setminus\{v\})\cup \{v'\},$ and the edge set of $G_2'$ is $E(G_2')=E(G_2\setminus v)\cup \{\{i,v'\}: \{i,v\}\in E(G_2)\}.$
Set $S'=S[x_{v'},y_{v'}], \ell_x=x_v-x_{v'}, \ell_y=y_v-y_{v'}.$ Then, as it was proved in \cite{RaRi}, $\ell_x,\ell_y$ is a regular sequence on $S'/J_{G'}$ and $S/J_G\cong S'/(J_{G'},\ell_x,\ell_y).$ Combining this with \cite[Theorem 4.7]{Sch} which holds in the graded case as well, it follows that for a decomposable graph  $G=G_1\cup G_2$, $J_G$ is sequentially Cohen-Macaulay if  and only if 
$J_{G_1}$ and $J_{G_2}$ are sequentially Cohen-Macaulay.
\end{Remark}

We recall a characterization of sequentially Cohen-Macaulay homogeneous ideals due to Goodarzi \cite{Goo}. 

Let $R=K[x_1,\ldots,x_n]$ be a polynomial ring  endowed with the standard grading and $I\subset R$ a homogeneous ideal. Let $I=\bigcap_{j=1}^s Q_j$ be a reduced primary decomposition of $I$ and set $P_j=\sqrt{Q_j}$ for $1\leq j\leq s.$ We denote by $I^{\langle i\rangle}$ the ideal 
\[
I^{\langle i\rangle}=\bigcap_{\dim R/P_j >i} Q_j.
\] for $-1\leq i\leq \dim R/I-1.$
By \cite[Proposition~16]{Goo}, we have the following.

\begin{Proposition}\label{Prop:Goo}
Let $I\subseteq R=K[x_1,\ldots,x_n]$ be a homogeneous ideal in a standard graded polynomial ring.\ Then $R/I$ is sequentially Cohen-Macaulay if and only if $\depth R/I^{\langle i\rangle}\geq i+1$ for all 
$0\leq i\leq \dim R/I-1.$
\end{Proposition}

\section{Sequentially Cohen-Macaulay binomial edge ideals of closed graphs}
\label{S:closed}

In  this section we characterize the  closed graphs whose binomial edge ideals are sequentially Cohen-Macaulay. In view of 
Remark~\ref{R:disconectSCM}, in what follows, we consider only connected indecomposable graphs. Note that a graph with a unique clique is a complete graph
hence with a Cohen-Macaulay binomial edge ideal.

To begin with, 
by using Proposition~\ref{Prop:Goo} we can prove that  indecomposable closed graphs with two maximal cliques have sequentially Cohen-Macaulay binomial edge ideals. 

\begin{Proposition}\label{Prop:2cliques}
Let $G$ be an indecomposable closed graph with $\Delta(G)=\langle F_1,F_2\rangle.$ Then $J_G$ is sequentially Cohen-Macaulay.
\end{Proposition}

\begin{proof}
Let us assume that $F_1=[1,b]$ and $F_2=[a,n]$ for some integers $1<a < b <n.$ Then $J_G$ has two cut sets, namely $\emptyset$ and $W=[a,b].$ We have 
\[
\dim \frac{S}{P_\emptyset(G)}=n+1 \mbox{ and } \dim \frac{S}{P_W(G)}=n+2-(b-a+1)=n+a-b+1\leq n.
\] Since $\depth S/P_\emptyset(G)=n+1$ because $P_\emptyset(G)$ is a Cohen-Macaulay ideal, by Proposition~\ref{Prop:Goo}, we only need to show that $\depth S/J_G\ge n+a-b+1.$ The equality $J_G=P_\emptyset(G)\cap P_W(G)$ induces the short exact sequence
\[ 0 \to \frac{S}{J_G}\to \frac{S}{P_\emptyset(G)}\oplus \frac{S}{P_W(G)} \to  \frac{S}{P_\emptyset(G)+P_W(G)}\to 0.\]
We observe that $P_\emptyset(G)+P_W(G)=(x_i,y_i: i\in W)+P_\emptyset(G)$  is a Cohen-Macaulay ideal with 
\[
\depth \frac{S}{P_\emptyset(G)+P_W(G)}=n-|W|+1=n-(b-a+1)+1=n+a-b \] 
\[<n+a-b+1=\depth\left(\frac{S}{P_\emptyset(G)}\oplus \frac{S}{P_W(G)}\right).
\] Therefore, by Depth Lemma, we get $\depth S/J_G=n+a-b+1.$
\end{proof}

In the next part, we consider indecomposable closed graphs with at least three maximal cliques. 
The main result of this section is the following.

\begin{Theorem}\label{Th:SCMClosed}

Let $G$ be an indecomposable closed graph with the clique complex $\Delta(G)=\langle F_1,F_2,\ldots,F_{r+1}\rangle, r\geq 2$
where $F_i=[a_{i-1},b_i]$ for $1\leq i\leq r+1$ and $1=a_0<a_1<\ldots<a_r,$ $b_1<\ldots<b_{r+1}=n.$
Then the following conditions are equivalent:
\begin{itemize}
\item [\emph{(a)}] $J_G$ is sequentially Cohen-Macaulay;
\item [\emph{(b)}] there exists $1\leq k\leq r$ such that $a_i=a_r -(r-i)$ for $k \leq i \leq r$ and $b_j=b_1+(j-1)$ for $1\leq j \leq k.$
\end{itemize}
\end{Theorem}

This section is devoted to the proof of the above theorem. 

\subsection{Proof of the implication  (a) $\Rightarrow$ (b).} 
Let $G$ be an indecomposable closed graph on the vertex set $[n].$ 
Since $G$ is indecomposable, notice that $|W_i|\ge 2$ for all $1\le r \le r$.

Hence, to prove this
implication is reduced to prove the following proposition:

\begin{Proposition}\label{Th:nonSCM} Let $G$ be an indecomposable closed graph.
Suppose there exist two connected cut sets $W_1=[a,b], W_2=[c, d]$ of $G$  with $a<d$,  $|W_1\setminus W_2| \ge 2$,   $|W_2\setminus W_1| \ge 2$ and there does not exist  a connected cut set $W=[e,f]$ of $G$   in $[a, d]$ such that   $f-e \le \max ( b-a, d-c)$.
 Then $S/J_G$ is not sequentially Cohen-Macaulay.
\end{Proposition}

Before proving this proposition, we need some preparation.

Let $G_1$ be the indecomposable closed graph  such that $\Delta(G_1)=\langle[1, b], [a, n]\rangle$.
Then we see that   $J_{G_1}= P_{\emptyset}(G)\cap  P_{W_1}(G)$.
Let $G_2$ be the closed graph  with $V(G_2)=[1,n]\setminus [c,d]$ such that 
 $\Delta(G_2)=\langle[1,  c-1], [a, n] \setminus [c,d]\rangle.$ (Here we consider $[a, n] \setminus [c,d]$ as a clique of $G_2$.)

A crucial tool is given by the next lemma:

\begin{Lemma}\label{L:lemma1} With the above notation, we have
\[J_{G_1}+ P_{W_2}(G)=J_{G_2}+(x_j, y_j: j \in W_2).\]
\end{Lemma}

\begin{proof}
The set of binomials
$
\{x_iy_j-x_jy_i : \  i,j \in [1, b] \ \mbox{\rm or  } i,j \in [a, n] \ \mbox{\rm or  }
  i,j \in [1,c-1] \ \mbox{\rm or  }  i,j \in [d+1, n]\}
\cup \{x_i ,y_i ; \  i\in W_2\}
$
is  a set of generators of $J_{G_1}+ P_{W_2}(G).$
Hence 
\begin{eqnarray*}
\{x_iy_j-x_jy_i : \  i,j  \in [a,n]\setminus W_2 \  \mbox{\rm or  }
  i,j \in [1,c-1] \}\cup \{x_i ,y_i ; \  i\in W_2\} 
\end{eqnarray*}
is   a set of generators of $J_{G_1}+ P_{W_2}(G).$
Each element of this  set belongs to $J_{G_2}+(x_j, y_j ; \ j \in W_2)$. Therefore, we get $J_{G_1}+ P_{W_2}(G)\subseteq J_{G_2}+(x_j, y_j ; j \in W_2).$

For the other inclusion, we observe that 
\begin{eqnarray*}
\{x_iy_j-x_jy_i ; \  i,j \in [1,c-1] \ \mbox{\rm or  } i,j \in [a, n] \setminus [c,d] \}
\cup \{x_i ,y_i; \   i\in W_2\}
\end{eqnarray*}
is  a set of generators of $J_{G_2}+(x_j, y_j ; \ j \in W_2).$
If   $i,j \in [1,c-1]$, then $x_iy_j-x_jy_i \in  P_{W_2}(G).$
If  $ i,j \in [a, n] \setminus [c,d] $, then  $x_iy_j-x_jy_i \in  J_{G_1}.$
Hence each of these  generators  belongs to $J_{G_1}+ P_{W_2}(G)$.
\end{proof}

Now we prove the above proposition:\newline
\vspace{.in}

\noindent\textit{Proof of Proposition  \ref{Th:nonSCM}.}
We may  assume $b-a+1=|W_1| \le |W_2|=d-c+1 .$
And with  $|W_1\setminus W_2| \ge 2$ we have
$d-b \ge c-a\ge 2$.

We localize $S/J_G$ at the prime ideal  $P=(x_i, y_i: i\in [a,d])+P_{\emptyset}(G)$.
Since $ \height P=2(d-a+1)+n-(d-a+1)+1=n+d-a+2$,
we have 
\begin{equation}\label{eq:dimSP}
\dim  S_P=n+d-a+2.
\end{equation}
By Proposition \ref{prop:cutsetsclosed}  in $S_P$ we have 
\[
(J_G)_P=(P_{\emptyset}(G))_P \cap (P_{W_1}(G))_P \cap (P_{W_2}(G))_P \cap  (P_{W_1 \cup W_2}(G))_P \cap  \cdots ,
\] 
 if $c \ge b+2$
and
\[
(J_G)_P=(P_{\emptyset}(G))_P \cap (P_{W_1}(G))_P \cap (P_{W_2}(G))_P \cap  \cdots ,
\] 
 if $c \le b+1$.
We have
\begin{eqnarray*}
\height P_{\emptyset}(G)& =&n-1, \\
\height P_{W_1}(G)&=&n+|W_1|-2,\\
\height P_{W_2}(G)&=&n+|W_2|-2.
\end{eqnarray*}

In addition, if $c \ge b+2$, we have 
\[\height P_{W_1 \cup W_2}(G)=n+|W_1|+|W_2|-3.\]

The intersection $ P_{\emptyset}(G)\cap  P_{W_1}(G)$ may be interpreted as a binomial edge ideal associated with an indecomposable closed graph with two maximal cliques that intersect in $W_1.$ Then by Proposition~\ref{Prop:2cliques}, it follows that 
$S/( P_{\emptyset}(G)\cap  P_{W_1}(G))$ is sequentially Cohen-Macaulay. On the other hand,  $S/P_{W_2}(G)$ is  Cohen-Macaulay.
By localizing at $P$ and using the above equalities together with (\ref{eq:dimSP}), we get
\begin{eqnarray*}
\depth (S/( P_{\emptyset}(G)\cap  P_{W_1}(G)))_P& =& \dim (S/( P_{W_1}(G)))_P=d-a-|W_1|+4=d-b+3, \\
\depth (S/P_{W_2}(G))_P &= &\dim (S/( P_{W_2}(G)))_P= d-a-|W_2|+4=c-a+3,\\
\dim (S/( P_{W_1 \cup W_2}(G)))_P&= &d-a-|W_1|-|W_2|+5=c-b+3 \  \mbox{\rm  if } c \ge b+2.
\end{eqnarray*}

In order to prove that $J_G$ is not sequentially Cohen-Macaulay, by \cite[Proposition~4.8]{Sch}, it is enough to show that  $(S/J_G)_P$  is not sequentially Cohen-Macaulay.

Set $k=c-a+2$.
We show that 
\[
\depth S_P/(J_G)_P^{<k>} \le k.
\] Then, by using Proposition~\ref{Prop:Goo}, we get the desired conclusion.

As $ \dim (S/( P_{W_1 \cup W_2}(G)))_P=  c-b+3\le c-a+2=k$ and 
by the assumption there does not exist  a cut set $W=[e,f]$ of $G$   in $[a, d]$ such that   $f-e \le \max ( b-a, d-c)$,
which means that  there does not exist  an  associated prime   $(P_W(G))_P$ of  $ (J_G)_P$  in $S_P$ such that 
\[ \dim (S/( P_{W}(G)))_P= d-a-(f-e+1)+4\ge d-a-(d-c+1)+3=c-a+3=k+1,\]
 we have 
\[
(J_G)_P^{<k>}=\cap _{I\in \Ass(J_G)_P, \  \dim S_P/I > k}I=(P_{\emptyset}(G))_P \cap (P_{W_1}(G))_P \cap (P_{W_2}(G))_P. 
\]

Considering the short exact sequence
\[
0\to S_P/(P_{\emptyset}(G)_P \cap P_{W_1}(G)_P \cap P_{W_2}(G)_P)\to  S_P/(P_{\emptyset}(G)_P \cap P_{W_1}(G)_P)\oplus  S_P/P_{W_2}(G)_P \to \] \[
\to  S_P/(P_{\emptyset}(G)_P \cap P_{W_1}(G)_P+P_{W_2}(G)_P)\to 0,
\]
by Depth Lemma, it is enough to show that  
\[
\depth S_P/(P_{\emptyset}(G)_P \cap P_{W_1}(G)_P+P_{W_2}(G)_P) \le k-1=c-a+1.
\]
By Lemma~\ref{L:lemma1}, it is enough to show that 
\[
\depth S_P/(J_{G_2}+(x_j, y_j : j \in W_2))_P\le c-a+1.
\]
But we have 
\begin{eqnarray*}
&&\depth S_P/(J_{G_2}+(x_j, y_j : j \in W_2))_P\\
&\le &\dim  S_P/(P_{[a,c-1]}(G_2)+(x_j, y_j : j \in W_2))_P\\
&=&\dim  S_P-2|W_2|-\mbox{\rm height } P_{[a, c-1]}(G_2)\\
&= &n+d-a+2-2|W_2|-\mbox{\rm height } P_{[a, c-1]}(G_2)\\
&= &n+d-a+2-2|W_2|-(n- |W_2|+|[a,c-1]|-2)\\
&= &d-a+2-(d-c+1)-(c-a)+2\\
&= & 3\le  c-a+1,
\end{eqnarray*}
since  $(P_{[a,c-1]}(G_2)+(x_j, y_j : j \in W_2))_P \ne S_P$ and $c-a \ge 2$.\qed

\vspace{.2in}

By the above proposition 
if $G$ is a closed set such that $S/J_G$   sequentially Cohen-Macaulay and  $W_1=[a_1, b_1], W_2=[a_2, b_2], W_3=[a_3, b_3]$    are   three connected cut sets  of $G$  with $a_1<a_2< a_3$,
then $|W_2| \le \max (|W_1|, |W_3| )$. 
Then we have the following corollary:

\begin{Corollary}
Let $G$ be an indecomposable closed graph with the clique complex $\Delta(G)=\langle F_1,F_2,\ldots,F_{r+1}\rangle, r\geq 2$
where $F_i=[a_{i-1},b_i]$ for $1\leq i\leq r+1$ and $1=a_0<a_1<\ldots<a_r,$ $b_1<\ldots<b_{r+1}=n.$
Set $W_i=F_i\cap F_{i+1}$  for $1\leq i\leq r.$  
 If $J_G$ is sequentially Cohen-Macaulay, then   there exists $1\leq k\leq r$ such that 
\begin{equation}\label{eq:cond}
|W_1|\geq |W_2|\geq \cdots \geq |W_k|\leq |W_{k+1}|\leq \cdots \leq |W_r|
\end{equation}
and   $a_i=a_r -(r-i)$ for $k \leq i \leq r$ and $b_j=b_1+(j-1)$ for $1\leq j \leq k.$
\end{Corollary}

\subsection{Proof of the implication $(b)\Rightarrow (a)$.}
Let $G$ be an indecomposable closed graph $G$ such that  $\Delta(G)=\langle F_1, F_2, \dots, F_{r+1}\rangle$.
Set $W_i:= F_i\cap F_{i+1}.$ Since $G$ is indecomposable, we have $|W_i| \ge 2$.
We assume that 
\[
F_1=[1, b], F_2=[a_1, b+1], F_3=[a_2, b+2],\dots  , F_k=[a_{k-1}, b+(k-1)], F_{k+1}=[a, b_{k+1}], \]
\[F_{k+2}=[a+1, b_{k+2}], \dots ,  F_{r}=[a+(r-k-1), b_{r}],  F_{r+1}=[a+(r-k), n].\]
where we set $a_k=a$ and $b_1=b.$ This is equivalent to saying that the connected cut sets of $G$ are
\begin{eqnarray*}
W_1=[a_1, b], W_2=[a_2, b+1], W_3=[a_3, b+2], \dots , W_k=[a, b+(k-1)], \\
W_{k+1}=[a+1,b_{k+1}], W_{k+2}=[a+2, b_{k+2}],\dots , W_{r}=[a+(r-k), b_{r}].
\end{eqnarray*} 

Let $H_1$ be the indecomposable closed graph  such that 
\[
\Delta(H_1)=\langle F_1, F_2, \dots,  F_{r-1}, F_r\cup F_{r+1}\rangle.
\]
Then, by the proof of Case 2 in Proposition~\ref{prop:cutsetsclosed}, we have
 \[\bigcap_{W:\mbox{ cut set of }G,\ W\not\supset W_r}P_W(G)=J_{H_1}.\]
Set  $H_2=G\setminus W_r$ and 
set $\ell=\max\{s: W_s \cap W_{r}= \emptyset \}$.
Then  
\[\Delta(H_2)=\langle F_1, F_2 \dots,  F_{\ell}, F_{\ell+1}\setminus W_r, F_{r+1}\setminus W_r\rangle.\]
By the first case in the proof of Proposition~\ref{prop:cutsetsclosed}, we have
\[\bigcap_{W: \mbox{ cut set of } G,\ W \supset W_r}P_W(G) =J_{H_2}+(x_j, y_j : j \in W_r).\]

Then \[
J_G^{<m>}= J_{H_1}^{<m>}\cap (J_{H_2}+(x_j, y_j : j \in W_r))^{<m>}.
\]

We also set $H_3$ to be the restriction of $G$ with $V(H_3)=[\min(W_{r-1}), n] \setminus W_r$ such that 
\[
\Delta(H_3)=\langle F_r\cup F_{r+1}\setminus W_r\rangle.
\] Therefore, we have $J_{H_3}=(x_iy_j-x_jy_i :i,  j \in F_r \cup F_{r+1}\setminus W_r).$

In the rest of this section, we will often use the notation $K[V(H)]$ for the polynomial ring $K[\{x_i,y_i: i\in V(H)\}].$

\begin{Lemma}\label{L:Lemma2} We have
\[J_{H_1}^{<m>}+ (J_{H_2}+(x_j, y_j : j \in W_r))^{<m>}\] \[=(J_{H_2}+(x_j, y_j : j \in W_r))^{<m>}+J_{H_3}\]
\end{Lemma}

\begin{proof}
($\subseteq$ )
 For any component  $P_W(H_2)$  of  $J_{H_2}^{<m>}$, where  \[J_{H_2}^{<m>}= \bigcap _{\dim K[V(H_2)]/P_W(H_2)> m}P_W(H_2),\]
$P_W(H_1)$  is a component of   $J_{H_1}^{<m>}$,
since $P_W(H_2)+(x_j, y_j : j \in W_r)=P_{W\cup W_r}(G)$ is a component of  $J_{G}^{<m>}$.
Since \[P_W(H_1) \cap K[\{x_i, y_i: i \in [1, n]\setminus F_r\cup F_{r+1}\}] =P_W(H_2)\cap K[\{x_i, y_i: i \in [1, n]\setminus F_r\cup F_{r+1}\}] ,\]
we have  \[J_{H_1}^{<m>}\cap K[\{x_i, y_i: i \in [1, n]\setminus F_r\cup F_{r+1}\}] \subseteq J_{H_2}^{<m>}\cap K[\{x_i, y_i: i \in [1, n]\setminus F_r\cup F_{r+1}\}].\]
Hence  \[J_{H_1}^{<m>} \subseteq (J_{H_2}+(x_j, y_j : j \in W_r))^{<m>}+(x_iy_j-x_jy_i :i,  j \in F_r \cup F_{r+1}\setminus W_r).\]

($\supseteq$ )
For any  $i,  j \in F_r \cup F_{r+1}\setminus W_r$ and  for any component  $P_W(H_1)$  of  $J_{H_1}^{<m>},$ we have
$ x_iy_j-x_jy_i \in P_W(H_1)$.
\end{proof}

\begin{Proposition} 
Let $G$ be an indecomposable closed graph with the clique complex $\Delta(G)=\langle F_1,F_2,\ldots,F_{r+1}\rangle, r\geq 2,$
where $F_i=[a_{i-1},b_i]$ for $1\leq i\leq r+1$ and $1=a_0<a_1<\ldots<a_r,$ $b_1<\ldots<b_{r+1}=n.$ Assume that there exists $1\leq k\leq r$ such that $a_i=a_r-(r-i)$ for $k \leq i \leq r$ and $b_j=b_1+(j-1)$ for $1\leq j \leq k.$ Then $S/J_G$ is a sequentially Cohen-Macaulay ideal.
\end{Proposition}

\begin{proof}
We proceed  by induction on $r$. We may assume $ |W_1|   \le  |W_r|$.
By induction,  we can assume that  $S/J_{H_1}$ is sequentially Cohen-Macaulay.
Then we have
\[
\depth S/J_{H_1}^{<m>}\ge m +1, \mbox{ for } m\geq 0.
\]

Let $H_4$ be the restriction of $G$  with $V(H_4)=[1, \min(W_{r-1})].$ Then  
\[
\Delta(H_4)=\langle F_1, F_2 \dots,  F_{\ell}, F_{\ell+1}\setminus W_r\rangle.
\]
By the inductive hypothesis,  $S/J_{H_4}$ is sequentially Cohen-Macaulay since $H_4$ is closed.
Hence we have 
\[
\depth K[V(H_4)]/(J_{H_4})^{<m-(n-\min(W_{r-1})+1)>}\ge m-(n-\min(W_{r-1})+1)+1.
\]
On the other hand, if we consider $H$ to be the closed graph with 
$\Delta(H)=\langle F_{r+1}\setminus W_r \rangle,$ whose binomial edge ideal is Cohen-Macaulay, we get
\[
S/(J_{H_2}+(x_j, y_j ; j \in W_r))\cong S/J_{H_4}\otimes_K S/J_{H},
\] which implies by \cite[Theorem 2.11]{STY} that $S/(J_{H_2}+(x_j, y_j ; j \in W_r))$ is sequentially Cohen-Macaulay, and, consequently, 
\begin{equation}\label{eq:eq3}
\depth S/(J_{H_2}+(x_j, y_j : j \in W_r))^{<m>}\ge m +1.
\end{equation}

Obviously, $J_{H_3}$ is Cohen-Macaulay with \[\dim K[V(H_3)]/J_{H_3}=\depth K[V(H_3)]/J_{H_3}=n-\min(W_{r-1})+2.\]


 Then \[S/((J_{H_2}+(x_j, y_j : j \in W_r))^{<m>}+J_{H_3})\] can be considered as  
\[K[V(H_4)]/(J_{H_4})^{<m-(n-\min(W_{r-1})>} \otimes_{K} K[V(H_3)]/J_{H_3} \mbox{ modulo } \{x-x', y-y'\},\] 
where $x=x_{\min(W_{r-1})} \in K[V(H_4)]$,  $y=y_{\min(W_{r-1})}\in K[V(H_4)]$,  $x'=x_{\min(W_{r-1})} \in K[V(H_3)]$,
$y'=y_{\min(W_{r-1})} \in K[V(H_3)]$.

\begin{Lemma}\label{L:regular}
The sequence $\{x-x', y-y'\}$ is regular on the tensor product 
\[K[V(H_4)]/(J_{H_4})^{<m-(n-\min(W_{r-1})>} \otimes_{K} K[V(H_3)]/J_{H_3}.\]
\end{Lemma}

\begin{proof}
We note that $x-x'$ is  regular  
since   $\Ass K[V(H_4)]/(J_{H_4})^{<m-(n-\min(W_{r-1})>} \otimes_{K} K[V(H_3)]/J_{H_3}$ contains prime ideals of
the form \[P=P_W(H_4) + (x_iy_j-x_jy_i :i,  j \in F_r \cup F_{r+1}\setminus W_r)\]
and $x-x' \not\in P$.

We show that $y-y'$ is a regular element of \[(K[V(H_4)]/(J_{H_4})^{<m-(n-\min(W_{r-1})>} \otimes_{K} K[V(H_3)]/J_{H_3})/(x-x').\]
Since it is known that  $y-y'$ is a regular element of \[(K[V(H_4)]/(J_{H_4})\otimes_{K} K[V(H_3)]/J_{H_3})/(x-x'),\] by \cite{RaRi},
it is enough to show that 
\[
\Ass (K[V(H_4)]/(J_{H_4})^{<m-(n-\min(W_{r-1})>} \otimes_{K} K[V(H_3)]/J_{H_3})/(x-x')
\subseteq \] \[
\Ass (K[V(H_4)]/J_{H_4}\otimes_{K} K[V(H_3)]/J_{H_3})/(x-x').
\]
In fact, it is enough to show 
\begin{eqnarray*}
&&\Ass (K[V(H_4)]/J_{H_4} \otimes_{K} K[V(H_3)]/J_{H_3})/(x-x')=\\
&&\bigcup _{P' \in \Ass (K[V(H_4)]/J_{H_4} \otimes_{K} K[V(H_3)]/J_{H_3})}
 \Min (K[V(H_4)] \otimes_{K} K[V(H_3)]/(P'+(x-x')),
\end{eqnarray*}
and
\[
\Ass (K[V(H_4)]/(J_{H_4})^{<m-(n-\min(W_{r-1})>} \otimes_{K} K[V(H_3)]/J_{H_3})/(x-x')=\]
\[\bigcup _{P' }
\Min (K[V(H_4)] \otimes_{K} K[V(H_3)]/(P'+(x-x'))
\] where 
$P' \in \Ass (K[V(H_4)]/(J_{H_4})^{<m-(n-\min(W_{r-1})>} \otimes_{K} K[V(H_3)]/J_{H_3})$.

By \cite[Corollary 3.2]{MVY},  we have
\[ \Ass (K[V(H_4)]/(J_{H_4}))^{<t>} \otimes_{K} K[V(H_3)]/J_{H_3})/(x-x')\]
\[\setminus  \bigcup _{P' \in \Ass (K[V(H_4)]/(J_{H_4})^{<t>} \otimes_{K} K[V(H_3)]/J_{H_3})}
\Min (K[V(H_4)] )^{<t>} \otimes_{K} K[V(H_3)]/(P'+(x-x'))\]
\[=\{P : P\  \mbox{\rm is a homogeneous ideal of } \dim (K[V(H_4)] \otimes_{K} K[V(H_3)])/P =s \ \mbox{\rm  such that }\]
\[((J_{H_4})^{<s-(n-\min(W_{r-1})+2)>}/(J_{H_4})^{<t>} \otimes_{K}K[V(H_3)]/J_{H_3})_{P}=0, \]
\[ (K[V(H_4)]/(J_{H_4})^{<s+1-n-\min(W_{r-1})+2)>} \otimes_{K} K[V(H_3)]/J_{H_3}))_{P} \ne 0,\]
\[\depth (K[V(H_4)]/(J_{H_4})^{<s+1-n-\min(W_{r-1})+2)>} \otimes_{K} K[V(H_3)]/J_{H_3}))_{P}=1,\   x-x' \in P  \}
=\emptyset.\]

Indeed, if $ (K[V(H_4)]/(J_{H_4})^{<s+1-n-\min(W_{r-1})+2)>} \otimes_{K} K[V(H_3)]/J_{H_3}))_{P} \ne 0$, then
\[\depth (K[V(H_4)]/(J_{H_4})^{<s+1-n-\min(W_{r-1})+2)>} \otimes_{K} K[V(H_3)]/J_{H_3}))_{P}=\] \[\dim (K[V(H_4)]/(J_{H_4})\otimes_{K} K[V(H_3)])/Q)_{P} \ge 2\]
by the sequentially Cohen-Macaulay property of \[K[V(H_4)]/(J_{H_4})^{<s+1-n-\min(W_{r-1})+2)>} \otimes_{K} K[V(H_3)]/J_{H_3},\]
where $Q \in \Ass  K[V(H_4)]/(J_{H_4})^{<s+1-n-\min(W_{r-1})+2)>} \otimes_{K} K[V(H_3)]/J_{H_3}) $ satisfies
\[\height Q= \bigheight(J_{H_4}^{<s+1-n-\min(W_{r-1})+2)>} +J_{H_3}).\] Here, $\bigheight(I)$ denotes the maximal height of the minimal prime ideals of the ideal $I.$
Hence, we know that
\[
\Ass (K[V(H_4)]/(J_{H_4})^{<m-(n-\min(W_{r-1})>} \otimes_{K} K[V(H_3)]/J_{H_3})/(x-x')
\subseteq \] \[
\Ass (K[V(H_4)]/(J_{H_4})\otimes_{K} K[V(H_3)]/J_{H_3})/(x-x'),
\]
and   $\{x-x', y-y'\}$ is a regular sequence on $K[V(H_4)]/(J_{H_4})^{<m-(n-\min(W_{r-1})>} \otimes_{K} K[V(H_3)]/J_{H_3}$.
\end{proof}

By the above lemma, we get
\begin{eqnarray*}
&&\depth S/((J_{H_2}+(x_j, y_j : j \in W_r))^{<m>}+J_{H_3})\\
&\geq & m-(n-\min(W_{r-1}+1)+1+n-\min(W_{r-1})+2-2=m.
\end{eqnarray*}

By Lemma~\ref{L:Lemma2}, we have the short exact sequence
\begin{eqnarray*}
&&0\to S/J_G^{<m>} \to  S/ J_{H_1}^{<m>}\oplus S/ (J_{H_2}+(x_j, y_j : j \in W_r))^{<m>}\\
&&\to  S/((J_{H_2}+(x_j, y_j : j \in W_r))^{<m>}+J_{H_3}) \to 0.
\end{eqnarray*}
Hence, by using  Depth Lemma, the inductive hypothesis on $J_{H_1}$ and inequality (\ref{eq:eq3}), we get 
\[
\depth S/J_G^{<m>} \ge m+1, \mbox{ for }m\geq 0. 
\]
Therefore $S/J_G$ is sequentially Cohen-Macaulay by Proposition~\ref{Prop:Goo}.
\end{proof}

\begin{Examples}
\begin{itemize}
	\item [1.] Let $G$ be the closed graph on the vertex set $[7]$ with the maximal cliques
	\[
	F_1=[1,3], F_2=[2,5], F_3=[3,6], F_4=[5,7].
	\] Then $J_G$ is not sequentially Cohen-Macaulay. Indeed, one easily sees that the connected cut sets $W_1=[2,3]$ and $W_2=[3,5]$ satisfy the conditions of Proposition~\ref{Th:nonSCM}.
	\item [2.] Let $G$ be the closed graph on the vertex set $[9]$  with the maximal cliques
	\[
	F_1=[1,3], F_2=[2,6], F_3=[3,7], F_4=[4,8], F_5=[5,9].
	\] Then $J_G$ is sequentially Cohen-Macaulay by Theorem~\ref{Th:SCMClosed}.
\end{itemize}
\end{Examples}

\section{Approximately  Cohen-Macaulay binomial edge ideals of closed graphs}
\label{S:Approx}

Closely related to sequentially Cohen-Macaulay modules is the concept of  approximately Cohen-Macaulay modules, introduced by Goto in \cite{G}. 

\begin{Definition} 
Let $M$ be a finitely generated graded module over a standard graded
$K$-algebra $R.$  Then $M$ is called \textit{almost Cohen-Macaulay} if $\depth M\geq \dim M-1.$
\end{Definition}

\begin{Definition}\label{defn:approxCM}
Let $M$ be a finitely generated graded module over a standard graded
$K$-algebra $R.$  Let $N_1,\ldots, N_r$ be a reduced primary decomposition of $M$ where $N_j$ is a $P_j$--primary module
for all $j.$ The $R$--module $M$ is called 
 \textit{approximately Cohen-Macaulay} if  $M$ is almost Cohen-Macaulay and
$M/\cap_{\dim R/P_j=d} N_j$ is a Cohen-Macaulay module. 
\end{Definition}

By \cite[Proposition 4.5]{Sch}, $M$ is approximately Cohen-Macaulay if and only if it is sequentially and almost Cohen-Macaulay.

The trees with approximately Cohen-Macaulay binomial edge ideals were charac\-terized in \cite{Z}.  In this section, we give the classification of closed graphs whose binomial edge ideals are approximately Cohen-Macaulay.

\begin{Remark}\label{R:almostCM}
Let $G$ be a closed graph with connected components $G_1,G_2,\ldots,G_c$. Then, by \cite[Theorem~2.17]{STY}, $J_G$ 
is almost (approximately) Cohen-Macaulay if and only if there exists $1\leq i\leq c$ such that $J_{G_i}$ is almost (approximately) Cohen-Macaulay and $J_{G_j}$ is Cohen-Macaulay for all $j\neq i, 1\leq j\leq c.$
\end{Remark}

We begin with connected closed graphs whose binomial edge ideals are almost Cohen-Macaulay. 
By Remark~\ref{R:almostCM}, in order to characterize the graphs $G$ whose   binomial edge ideals are almost Cohen-Macaulay,   we may consider $G$ indecomposable. 

\begin{Theorem}\label{thm:almostCMclosed}
Let $G$ be an indecomposable closed graph such that $J_G$ is not Cohen-Macaulay. Then $J_G$ is almost Cohen-Macaulay if and only if $G$ satisfies one of the following conditions:
\begin{itemize}
	\item [\emph{(a)}] The clique complex $\Delta(G)$ has just two facets, which intersect in two vertices;
	\item [\emph{(b)}] The clique complex $\Delta(G)$ has just three facets, which are of the form $F_1=[1,b], F_2=[b-1,b+1], F_3=[b,n]$ for some 
	$3\leq b\leq n-2;$
	\item [\emph{(c)}] The clique complex $\Delta(G)$ has just four facets, which are of the form 	$F_1=[1,b], F_2=[b-1,b+1], F_3=[b,b+2], F_4=[b+1,n]$ for 
	some $3\leq b\leq n-3.$
\end{itemize}
\end{Theorem}

\begin{proof}
Let $J_G$ be almost Cohen-Macaulay. Then $\depth S/J_G\geq \dim S/J_G-1=n.$ Since $G$ is indecomposable, $J_G$ is not Cohen-Macaulay, thus 
$\depth S/J_G=n.$ Let $F_1=[a_1,b_1],F_2=[a_2,b_2],\ldots,F_r=[a_r,b_r]$ be the facets of $\Delta(G)$ with $1=a_1<a_2<\cdots <a_r<b_r=n.$
Since $\dim S/P_W(G)\geq \depth S/J_G$ for every cut-set $W,$ we get $n-|W|+c(W)\geq n,$ that is, $c(W)\geq |W|.$ In particular, this implies that every intersection $F_i\cap F_{i+1}$ has cardinality $2$ for all $1\leq i\leq r-1.$ If there exists $1\leq i\leq r-2$ such that $\max(F_i\cap F_{i+1})+1< \min(F_{i+1}\cap F_{i+2}),$  then 
$W:=(F_i\cap F_{i+1})\cup(F_{i+1}\cap F_{i+2})$ is a cut-set of $G$ with $|W|=4$ and $c(W)=3$. Therefore, $\dim S/P_G(W)=n-1,$ contradiction. Therefore, for all  $1\leq i\leq r-2,$ we have $\max(F_i\cap F_{i+1})+1\geq  \min(F_{i+1}\cap F_{i+2}).$

Let us first assume that the number of maximal cliques of $G$ is $\geq 5.$ In this case, we choose, for example, 
$W=(F_1\cap F_2)\cup (F_4\cap F_5).$ Then $\max(F_1\cap F_2)=b_1$ and $\min (F_4\cap F_5)=a_5.$ If  $a_5>b_1+1,$ then we get 
$|W|=4$ and $c(W)=3$ which yields a contradiction. Thus, we must have $a_5\leq b_1+1.$ But $a_5\geq a_2+3=b_1-1+3=b_1+2,$  again a contradiction. 
Therefore, $\Delta(G)$ has at most 4 facets. 

If $\Delta(G)$ has two facets, say $F_1=[1,b], F_2=[b-1,n],$ then $\depth S/J_G=n$ by the proof of Proposition~\ref{Prop:2cliques}.

Let $\Delta(G)$ with three facets, say $F_1=[1,b_1], F_2=[b_1-1, b_2],$ and $F_3=[b_2-1,n]$ for some  $3\leq b_1 < b_2\leq n-1.$ On the other hand, as we have seen above, we have $\max(F_1\cap F_{2})+1\geq  \min(F_{2}\cap F_{3})$ which implies that $b_1+1\geq b_2-1,$ or
$b_2\leq b_1+2.$ Hence, we have either $b_2=b_1+2$ or $b_2=b_1+1.$ 

If $b_2=b_1+2,$ we set $b_1=b$ and get 
$F_1=[1,b], F_2=[b-1,b+2], F_3=[b+1,n].$ It is easily seen that $J_G=P_\emptyset(G)\cap P_{\{b-1,b\}}(G)\cap P_{\{b+1,b+2\}}(G)$ where the last two prime ideals have dimension $n.$ We may write
\[
J_G=J_{G_1}\cap P_{\{b+1,b+2\}}(G)
\]
where $G_1$ is the closed graph with the maximal cliques $[1,b],[b-1,n].$ From above, we know that $\depth S/J_{G_1}=n.$
Let us consider the exact sequence 
\begin{equation}\label{eq:2nd}
0\to \frac{S}{J_G} \to \frac{S}{J_{G_1}}\oplus \frac{S}{P_{\{b+1,b+2\}}(G)}\to \frac{S}{P_{G_1}+P_{\{b+1,b+2\}}(G))}\to 0.
\end{equation}
We observe that $P_{G_1}+P_{\{b+1,b+2\}}(G)=(x_{b+1},y_{b+1},x_{b+2},y_{b+2})+ J_L$ where $J_L$ is the binomial edge ideal of a closed graph on $n-2$ vertices  with the maximal cliques $[1,b]$ and $[b-1,n]\setminus\{b+1,b+2\}.$ As the maximal cliques intersect in two vertices, we derive that 
$\depth S/(P_{G_1}+P_{\{b+1,b+2\}}(G)))=n-2.$ By Depth Lemma applied on sequence (\ref{eq:2nd}), we obtain 
$\depth S/J_G=n-1$, a contradiction to our assumption on $J_G.$ Thus, we only have to analyze the case $b_2=b_1+1.$ We set $b_1=b.$ Then the maximal cliques of $G$ are $F_1=[1,b], F_2=[b-1,b+1], F_3=[b,n]$ and $J_G=P_\emptyset(G)\cap P_{\{b-1,b\}}(G)\cap 
P_{\{b,b+1\}}(G)$ where the last two prime ideals have dimension $n.$
Again, we may write 
\[
J_G=J_{G_1}\cap P_{\{b,b+1\}}(G)
\]
where $G_1$ is the closed graph with the maximal cliques $[1,b],[b-1,n]$ and consider the exact sequence 
\begin{equation}\label{eq:3rd}
0\to \frac{S}{J_G} \to \frac{S}{J_{G_1}}\oplus \frac{S}{P_{\{b,b+1\}}(G)}\to \frac{S}{P_{G_1}+P_{\{b,b+1\}}(G)}\to 0.
\end{equation}
But in this case, $P_{G_1}+P_{\{b,b+1\}}(G)=(x_by_b,x_{b+1},y_{b+1})+J_N$ where $J_N$ is the binomial edge ideal of the closed graph on 
$n-2$ vertices with the maximal cliques $[1,b-1]$ and $[b-1,n]\setminus\{b,b+1\}.$ The ideal $J_N$ is Cohen-Macaulay since the two cliques intersect in just one vertex. Therefore, $\depth S/(P_{G_1}+P_{\{b,b+1\}}(G))=n-1.$ With Depth Lemma on 
sequence (\ref{eq:3rd}), we get $\depth S/J_G=n.$ 

Finally, for the case when $G$ has $4$ maximal cliques, we may proceed in a simi\-lar way to derive that the maximal cliques of $G$ are 
\[F_1=[1,b], F_2=[b-1,b+1],  F_3=[b,b+2],\mbox { and } F_4=[b+1,n].\] Then the minimal prime ideals of $J_G$ are 
$P_{\emptyset}(G), P_{\{b-1,b\}}(G),P_{\{b,b+1\}}(G),P_{\{b+1,b+2\}}(G),$ the last three being of dimension $n.$ We may write 
\[
J_G=J_{G_2}\cap P_{\{b+1,b+2\}}(G)
\]
where $G_2$ is the closed graph with the maximal cliques $[1,b],[b-1,b+1],[b,n].$ From our previous discussion, we derive that 
$\depth S/J_{G_2}=n$  and next consider the following exact sequence
\begin{equation}\label{eq:4th}
0\to \frac{S}{J_G} \to \frac{S}{J_{G_2}}\oplus \frac{S}{P_{\{b+1,b+2\}}(G)}\to \frac{S}{P_{G_2}+P_{\{b+1,b+2\}}(G)}\to 0.
\end{equation}
in order to get $\depth S/J_G\geq n.$
 But we also have $\depth S/J_G\leq n.$ Therefore, we get equality. 
\end{proof}

The following examples illustrate the case (b) in the above theorem.

\begin{Examples}
\begin{itemize}
	\item [1.] Let $G$ be the closed graph on the vertex set $[7]$  with the maximal cliques $[1,4], [3,6],[5,7].$ Then $J_G$ is not almost Cohen-Macaulay since it does not satisfy condition (b) of Theorem~\ref{thm:almostCMclosed}. Note that according to Theorem~\ref{Th:SCMClosed}, $J_G$ is not sequentially Cohen-Macaulay as well.
	\item [2.] Let $G$ be the closed graph on the vertex set $[7]$  with the maximal cliques $[1,4], [3,5],[4,7].$ Then $G$ satisfies condition (b) of 
	Theorem~\ref{thm:almostCMclosed}, thus $J_G$ is almost Cohen-Macaulay. The graph $G$  satisfies also the condition (b) in Theorem~\ref{Th:SCMClosed}, therefore, 
	$J_G$ is sequentially Cohen-Macaulay.
\end{itemize}
\end{Examples}

Let $G$ be a connected closed graph and let $G=G_1\cup G_2\cup\cdots\cup G_r$ be the decomposition of $G$ into indecomposable subgraphs. 
By Remark~\ref{R:decompose}, every component of $G$ whose binomial edge ideal is Cohen-Macaulay is a maximal clique of $G.$

\begin{Corollary}\label{cor:almostCMclosed}
Let $G$ be a connected closed graph such that $J_G$ is not Cohen-Macaulay. Then $J_G$ is almost Cohen-Macaulay if and only if there exists a unique 
$1\leq i\leq r$ such that $G_i$ is of one of the forms from Theorem~\ref{thm:almostCMclosed} and $G_j$ is a maximal clique of $G$ for all $j\neq i.$ 
\end{Corollary}

\begin{Theorem}\label{thm:approxCMclosed}
Let $G$ be a closed graph. Then $J_G$ is approximately Cohen-Macaulay if and only if $J_G$ is almost Cohen-Macaulay.
\end{Theorem}

\begin{proof}
As we have already discussed, we may consider $G$ a connected closed graph. The direct implication follows by 
\cite[Proposition 4.5]{Sch}. For the converse, let $G=G_1\cup G_2\cup\cdots\cup G_r$ be the decomposition of $G$ into indecomposable subgraphs. Since $J_G$ is almost Cohen-Macaulay, by Corollary~\ref{cor:almostCMclosed}, one of the components has an  almost Cohen-Macaulay binomial edge ideal  and all the others are maximal cliques of $G.$ Therefore, by \cite[Proposition~4.5]{Sch} it is enough to show that every graph $G$ of one of the forms from Theorem~\ref{thm:almostCMclosed} is sequentially Cohen-Macaulay. 
 But this is clearly true, since in all cases, the graph satisfies condition (b) from Theorem~\ref{Th:SCMClosed}.
\end{proof}

\section*{acknowledgment}
We thank an anonymous reviewer for the careful reading of our manuscript and the  insightful comments and suggestions.

\subsection*{Data availability statement} Our manuscript has no associated data.

\end{document}